\documentclass[12pt]{article}
\usepackage[all]{xy}
\usepackage{wrapfig}
\usepackage{multicol}

\usepackage{epic}
\usepackage{amsmath}[1996/11/01]
\usepackage{amssymb,amsthm,amsfonts,latexsym,epsfig,graphics}
 \def\beql#1#2\eeql{\begin{equation}\label{#1}#2\end{equation}}

\advance \textheight by -1 \topmargin
\advance \textheight by -1 \headheight
\advance \textheight by -1 \headsep
\advance \textheight by -2 \footskip
\vsize = \textheight
\hsize = \textwidth
\newcommand{\void}{0 }
\newcommand{\emptyword}{\epsilon}

\DeclareMathOperator{\PP}{PPerm}
\DeclareMathOperator{\PM}{Inj}
\DeclareMathOperator{\pos}{pos}

\DeclareMathOperator{\Sym}{Sym}

\newtheorem{theorem}{Theorem}[section]

\newcommand{\bew}{\noindent\underline{Proof.}\ }

\newtheorem{remark}[theorem]{Remark}
\newtheorem{lemma}[theorem]{Lemma}

\newtheorem{defn}[theorem]{Definition}

\newcommand{\disj}{\stackrel{.}{\cup}}

\newcommand{\N}{{\mathbb{N}}}

\newcommand{\eb}{\phantom{zzz}\hfill{$\square $}\smallskip}

\newcommand{\cB}{{\mathcal B}}
\newcommand{\cD}{{\mathcal D}}
\newcommand{\cF}{{\mathcal F}}
\newcommand{\cG}{{\mathcal G}}

\newcommand{\cJ}{{\mathcal J}}
\newcommand{\cM}{{\mathcal M}}
\newcommand{\cZ}{{\mathcal Z}}
\newcommand{\cj}{{\sf j}}

\newcommand{\PCT}{{\rm PCT}}
\newcommand{\II}{{\Lambda}} 
\newcommand{\ii}{{\lambda}} 
\newcommand{\NN}{{N_{\max}}} 
\newcommand{\Nht}{{N_{ht}}} 

\title{A method for building permutation representations of finitely presented groups}
\author{Gabriele Nebe, Richard Parker, and Sarah Rees} 
\begin{document}
\maketitle
\textsc{
Lehrstuhl D f\"ur Mathematik, RWTH Aachen University,
52056 Aachen, Germany}

\emph{E-mail address}{:\;\;}\texttt{nebe@math.rwth-aachen.de}
\medskip

\textsc{
70 York St. Cambridge CB1 2PY, UK}

\emph{E-mail address}{:\;\;}\texttt{richpark54@hotmail.co.uk}
\medskip

\textsc{
School of Mathematics, University of Newcastle, Newcastle NE1 7RU, UK}

\emph{E-mail address}{:\;\;}\texttt{Sarah.Rees@ncl.ac.uk}
\medskip

{\sc Abstract.} 
We design an algorithm to find
certain partial permutation representations of a finitely 
presented group $G$ (the bricks) that may be combined 
to a transitive permutation representation of $G$ (the mosaic) on the disjoint union. 

{\sc Keywords.} finitely presented groups; computational methods; permutation representation 

{\sc 2000 MSC: 20F05; 20B40} 

\section{Introduction}
\label{sec:intro}

For a group $G$ described by a finite presentation $\langle x_1,\ldots,x_n \mid r_1,\ldots,r_m\rangle$, it is often useful to know which
groups arise as quotients (equivalently, as homomorphic images).

Two well studied techniques
investigate the occurrence of a given finite permutation group $H$
as an image of $G$,
namely the `low index subgroup algorithm'  and Holt's `permutation image' program;
both are widely described in the literature, for example in \cite{Holt},
and well developed as algorithms.
For any given pair of groups $G,H$ as above, either of these two
techniques can be used to give a definite answer to the
question of whether or not $G$ has $H$ as a finite homomorphic image.
But each technique requires a systematic search, and this can be slow.

Using a commutative algebra technique, Plesken developed an 
algorithm  to classify all homomorphic images isomorphic to a linear group of degree 2 or 3,
see \cite{PleFab} and \cite{L2Q}.
This algorithm also finds certain
infinite homomorphic images; 
others can be found using Plesken's soluble quotient algorithm \cite{SOLQ}.

An alternative method, used (so far) to find finite images of $G$,
constructs permutation images as coset diagrams.
This technique was used by Stothers \cite{Stothers1,Stothers2}
to study subgroups of the modular
group and Hurwitz group (and attributed by him to Conway, and Singerman \cite{Singerman}) and subsequently extensively developed by Higman and others
\cite{Conder1,Conder2,Everitt1}. 

Higman suggested the use of coset diagrams to study images of the
modular group, the Hurwitz group and other triangle groups. Higman's ideas were
developed by Conder, Mushtaq, Servatius, Everitt
 and Kousar in \cite{Conder1,Conder2,Everitt1,Everitt2,Kousar,MushtaqServatius}
; in \cite{Everitt2} the method produced a positive proof to Higman's
conjecture that every Fuchsian group has all but finitely many alternating
groups amongst its homomorphic images; the result was then extended to
non-Euclidean crystallographic groups in Everitt's student Kousar's thesis
\cite{Kousar}.

The coset diagram technique is one that works well for a group $G$ with a small presentation.
An action of such a group $G$ on the cosets of a subgroup of index $n$ is
described by a diagram with $n$ vertices and directed arcs labelled by generators between
vertices, that satisfies certain conditions. Properties of the associated
homomorphic image can easily be deduced from properties of the diagram. Various
techniques allow diagrams to be joined together to produce actions of
higher degrees \cite{Stothers2,Conder1,Everitt2}, or to be modified to produce actions of other groups with
presentations similar to those of $G$ \cite{Everitt2,HoltPlesken}.

Our article is motivated by a study of the coset diagram technique,
and in particular the methods used to combine diagrams.

We suppose that $G=\langle X \mid R \rangle$ is a finitely presented group.
Our aim is to describe how to construct 
a transitive permutation representation $\phi:G \rightarrow \Sym(\cM)$ (the {\bf mosaic})
for the group $G$
as a combination of (partial) 
 permutation representations (the {\bf bricks}) 
$P_B: G \rightarrow \Sym(\Omega_B)$ 
of a certain form.
To combine the bricks and form a mosaic,
we need to be able to `jump' from one brick to another,
at points to which `pieces of cement' (from a set $C$)  are attached. {\bf Jump data} needs to be designed by the user 
to determine via which generators we are allowed
to jump, as well as conditions under which we must `stay'
within a brick; given a set $\cJ$ of jump data, we define an associated 
groupoid $\cF(\cJ)$. 
Assuming compatibility between $\cJ$ and $R$, the set
$R$ of relators for $G$ determines a set ${\frak R}$ of words over the elements
of $\cJ$ that define products in $\cF(\cJ)$.
We define the {\bf jump groupoid} $\cG(\cJ,{\frak R})$
as the quotient of $\cF(\cJ)$ by the normal closure of $\langle \frak{R}\rangle$.

The bricks that are used to build the mosaic
are, in effect, special partial permutation representations with
{\bf handles} 
 attaching cement to the underlying set  in a way that is
compatible with the jump data.
Section~\ref{sec:brickfinder} describes an algorithm that takes as
input a presentation for a group and associated jump data,
and outputs all possible bricks
up to some given degree. 

In section~\ref{sec:mosaic}, we assume that we have compatible jump data
for $G$  and a set $\cB$ of bricks. Then 
the {\bf construction instruction} defines a groupoid action of the 
jump groupoid $\cG(\cJ,{\frak R})$ on the 
set of all handles in all bricks of $\cB $ and tells us how  
to construct the associated mosaic, which is a permutation 
representation of $G$ on the disjoint union $\bigcup _{B\in \cB} \Omega _B$.

The final section of the paper, Section~\ref{sec:examples}
illustrates the method and the notions introduced in Section~\ref{sec:theory} 
by elaborating some easy examples. 
Among other things we show how to use the new technique to 
prove that a certain finitely generated group has 
almost all alternating groups as a quotient. 

\section{The construction of mosaics} 
\label{sec:theory}

\subsection{Jump data for $F(X)$} 
\label{sec:jumpdata}

Let $X$ be a finite set, closed under inverses. 
Instead of working with a free group on the set $X$ we will 
work 
with an {\bf inv-tab group} F(X), which we define to be the free group with its table of inverses adjoined. 
Note that if $x=x^{-1}$, we obtain 
the relation $x^2=1$ in $F(X)$. 
Where $G= \langle X \mid R \rangle$ is a finitely presented group, $G$
arises as a quotient of
the inv-tab group $F(X)$. The jump data that we use is associated with $F(X)$
rather than with $G$, but needs to be consistent with $R$.

We define {\bf jump data} for the inv-tab group $F(X)$ to
consist of a triple $\cJ=(C,J,S)$ of sets, and associated
maps.

$C$ is  a finite set of {\bf pieces of cement}, equipped with
an involution $\overline{\phantom{c}}$ on $C$, a map $\xi : C \to X $ satisfying 
$\xi(\overline{c}) = \xi (c) ^{-1 },$ and a bijection $\cj$ from $C$ to the
set of triples 
\[ J=\cj (C)=\{\cj(c):=(c,\xi(c),\overline{c}) : c \in C \};\]
we call $J$
the set of {\bf jumps}
, and we call a generator $x$ {\bf cementable} if it is an image $\xi(c)$
of a piece of cement.
Note that $J$ admits an involution that swaps $\cj(c) = (c,x,\overline{c})$ and
$\cj(\overline{c}) =
(\overline{c},x^{-1},c)$. 

$S$, the set of {\bf stays}, is a subset of
$C \times F(X) \times C$ that
is required to be closed under inverses, that is, 
\[(c , {\sf w} , c' ) \in S \Rightarrow (c',{\sf w}^{-1} , c) \in S .\]
We require that $S$ satisfies the following {\bf consistency condition}: 
\begin{quote}
If $(c , {\sf w_1} , c_1 )$ and $(c,{\sf w_2},c_2)$ are distinct stays in $S$, then ${\sf w_1}$ and ${\sf w_2}$ are incomparable (that is, neither is a prefix, empty or otherwise, of the other).
\end{quote}
This consistency condition is needed so that whenever we evaluate the word ${\sf w_1}$ at
some point of our brick cemented by $c$, then we never reach another cement point until the end of 
${\sf w}_1$ when we reach a point cemented by $c_1$. 

Let $\sim_S $ be the equivalence relation on $C$ that is the reflexive and
transitive closure of
the binary symmetric relation relating $c,c'$ whenever
some triple $(c,{\sf w},c')$ is in $S$.
We write $C= C_1\disj \ldots \disj C_\Nht $ as the disjoint union of the 
equivalence classes, $C_i$. 
We see the elements of each class $C_i$ as vertices of a 
connected, directed, labelled graph, which we call a {\bf handle type}. 
The handle type $H_i$ has vertex set $C_i$ and edges those stays 
$(c,{\sf w},c') $ for which $c,c'\in C_i$. 

We define an inv-tab groupoid ${\cF} (\cJ)$
 on the set 
$$J = \{ \ _n \cj (c) _m : c \in C_n ,\overline{c} \in C_m , 1\leq n,m \leq \Nht \} $$
(the subscript notation identifies $n$ as the source and $m$ as the target 
of $\cj (c)$),
consisting of all words $\cj(c_1) \ldots \cj(c_s)$ with 
$\overline{c_i}  \sim_S  c_{i+1} $ for all $1\leq i < s$.
The inverse of the generator $\cj(c)$ is $\cj(c)^{-1} = \cj(\overline{c}) $.

\subsection{The groupoid relators}  
\label{sec:relators}

Now let $R\subset F(X)$ be a subset of
the inv-tab group on $X$ and 
 $$G:=\langle X \mid R \rangle $$ the finitely presented group defined by $R$. 

We will assume without loss of generality that the  set $R$ 
of relators is closed under inversion and rotation, 
(so $R=R^{-1}$ and 
if $x_1x_2\ldots x_n \in R $ then also $x_nx_1x_2\ldots x_{n-1} \in R $).
If $r=(x_1x_2\ldots x_t)^{m}$ is a perfect power of a primitive word $r':= x_1x_2\ldots x_t$,
then its 
rotational equivalence class contains exactly $t$ elements
$$\circ ((x_1x_2\ldots x_t)^{m}) = \{ r_k := (x_{k} x_{k+1} \ldots x_t x_1\ldots x_{k-1})^m : 1\leq k \leq t \} .$$
We call $\{1,\ldots , t\} $ the different positions of $r$,
and, for $x\in X$, we define 
$$\pos (x,r) := \{ 1\leq i \leq t : x_i = x\} $$
the set of positions in $r'$ in which the generator $x$ occurs.

Suppose now that jump data $\cJ= (C,J,S)$ is given.
\begin{defn}
We say that $\cJ$ is {\bf compatible with $R$} if for each $r \in R$, for each $c \in C$,
and for each $k \in \pos(\xi(c),r)$, the $k$-th cyclic shift $r_k$ of $r$ can
be factorised as a product
$x_{i_1}{\sf w}_1x_{i_2}{\sf w}_2\cdots x_{i_s}{\sf w}_s$
where, 
for each $1\leq j \leq s $, 
we have $x_{i_j} = \xi(c_j)$ for some $c_j \in C$, $c_1=c$,
and $ (\overline{c_j} , {\sf w}_j, c_{j+1} ) \in S $
(where addition is taken $\bmod\,s$). 
We say that such a factorisation of $r$ is {\bf compatible with $\cJ$}
\end{defn}
Note that the consistency condition on $S$ ensures that, whenever a cyclic shift of a relator
admits a factorisation as above, then it is the unique such factorisation associated with $c$ and $k$.
Note also that the factorisation of a relator $r$ that is a proper power $(r')^m$ 
may well not be a power of a factorisation of $r'$.
We shall need $\cJ$ to be compatible with $R$ in order for our construction to work.

So now, assuming that $\cJ$ is compatible with $R$,
then for each $r \in R$, and each $r_k \in o(r)$ as above,
we define a product
${\frak r}_k : = \cj(c_1) \cdots \cj(c_s)$ in the inv-tab groupoid.
We let ${\frak R}$ be the set of all such products (for all $r \in R$),
and we define
the {\bf jump groupoid} associated with $G=\langle X \mid R\rangle $ and the
jump data $\cJ$
to be the
quotient
$${\cG}(\cJ,{\frak R}) := \cF (\cJ) / \langle {\frak R}
 \rangle ^N $$
of the inv-tab groupoid $\cF(\cJ)$.
Note that for all elements of ${\frak R}$, source and target 
coincide, so these live in subgroups of the groupoid, in which we
take the normal closure $\langle {\frak R} \rangle ^N$
  of the group generated by these elements of 
${\frak R}$ to obtain $\cG (\cJ ,{\frak R})$ as a quotient of 
$\cF(\cJ)$.

\subsection{Bricks} 
\label{sec:bricks}
Now suppose that $G$ and jump data $\cJ=(C,J,S)$ for $X$ are specified.
A brick is, essentially, a `cemented' partial permutation representation of $G$
that is compatible with the jump data, and which contains embedded images
of handles.

For  sets $\Omega, \Omega '$, we 
let $\PM(\Omega,\Omega ')$ be the set of injective maps from $\Omega$ to
$\Omega'$. 
We define a partial permutation $\pi $ 
of a set $\Omega$ to be a mapping from
$\Omega $ to $\Omega \cup \{ \void \} $ such that for all $\omega \in \Omega $
the preimage
$\pi ^{-1} (\{ \omega \} )$ has at most one element;
we write $\PP(\Omega)$ for the set of such partial permutations.
A map $\Pi:X \rightarrow \PP(\Omega)$
defines a partial permutation representation of $G$ on a set $\Omega$
if for all
$x_1,\ldots , x_s \in X$ and all $\omega \in \Omega $ 
such that 
\[ \omega\cdot \Pi(x_1),\,\omega\cdot \Pi(x_1) \Pi (x_2),\ldots ,
\omega\cdot \Pi(x_1) \cdots \Pi (x_s) =: \omega\cdot \Pi({\sf w})\]
are all defined as elements of $\Omega$,
in which case we 
say that a word ${\sf w} = x_1\ldots  x_s$ {\bf stays within $\Omega$ 
from the point $\omega \in \Omega $}, 
then $\omega\cdot \Pi({\sf w})$ only depends on ${\sf w} \in G$. 
The partial permutation representation $\Pi $ is called {\bf transitive},
if for all $\omega, \omega' \in \Omega $ there is some word
${\sf w}$ staying in $\Omega $ from $\omega $, such that $\omega \cdot \Pi ({\sf w}) = \omega '$.

\begin{defn} \label{defn:brick} 
Suppose that $\Pi:X \rightarrow \PP(\Omega)$ defines a  transitive
partial permutation representation of $G$. 
\begin{itemize}
\item[(B1)] 
Given a handle type $H_i$ of $\cJ$ with vertex set $C_i$, a map
$\theta: C_i \rightarrow \Omega$
is called a {\bf handle}  of type $H_i$ for $\Pi$, provided that\newline
(a) for all $c \in C_i$, $\theta (c) \cdot \Pi (\xi (c))  = \void $ 
and\newline
(b) for all edges $(c,{\sf w }, c') $ of $H_i$,
${\sf w}$ stays within $\Omega $ from $\theta(c)$,
and 
$\theta (c) \cdot \Pi ({\sf w}) = \theta (c' ) .$
\item[(B2)] 
Given a set of non-negative integers $\{h_i: 1\leq i \leq \Nht\}$,
let $\theta_{i,j}$ be
a handle of type $H_i$ for $\Pi$, for each $i=1,\ldots,\Nht,j=1,\ldots, h_i$.
Then, where $\hat{C}:= \bigcup_{i\in \{1,\ldots,\Nht\}} (C_i \times \{1,\ldots,h_i\})$,

define 
$\hat{\theta}:\hat{C} \rightarrow \{ (\omega , x) \mid \omega\cdot\Pi (x) = \void \}$ 
by
\[\hat{\theta}(c,j)=(\theta_{i,j}(c),\xi(c))\mbox{ when }c \in C_i.\]
We say that $B=(\Pi,\hat{\theta})$ is a {\bf brick} provided that the map
$\hat{\theta}$ is a bijection.
We call the elements of $\hat{C}$ {\bf cement points}.
\end{itemize}
\item[(B3)]
We extend $\Pi$ to a map 
$P : X \to 
\PM(\Omega , \Omega\disj \hat{C} )$ via
\[ \omega \cdot P(x):=\left \{ \begin{array}{ll} 
\omega \cdot \Pi(x) & \mbox{if } \omega\cdot \Pi(x)\neq\void\\
\hat{\theta}^{-1}(\omega,x) &\mbox{otherwise.}\end{array}\right.\]
The brick $(\Pi,\hat{\theta})$ is completely determined by the map $P$,
and so we may also call the map $P$ a brick, and write $P=(\Pi,\hat{\theta})$.
\end{defn} 

\begin{remark}
\label{rem:brickdefn}
(a) 
One may visualise a brick as a table with rows indexed by the elements of 
$\Omega $ and columns indexed by the elements of $X$. 
The position $(\omega,x) $ of the table is filled by $\omega \cdot P(x)$. 
\\
(b) If ${\sf w} = x_1\cdots x_s $ stays within $\Omega $ from the point 
$\omega $, then we may evaluate 
$$\omega P({\sf w} ) = \omega P(x_1) P(x_2) \cdots P(x_s) .$$ 
\,\\
(c) We note that the set 
\[ \{ \omega \cdot P(x): \omega \in \Omega, x \in X \} \]
can be expressed as the disjoint union of $\Omega$ and $\hat{C}$ 
(the domain of the bijection $\hat{\theta}$), and that the  
codomain of $\hat{\theta}$ can be expressed as 
$\{ (\omega,x):\omega \in \Omega, x \in X, \omega \cdot P(x) \not \in \Omega \}$
\end{remark}

The following is essentially a restatement of property (B1) of Definition~\ref{defn:brick}, and will be useful later.

\begin{lemma}
\label{lem:brick}
Suppose that $P=(\Pi,\hat{\theta})$ is a brick, as defined above,
that $\omega \in \Omega$ is within the image of a handle $\theta_{i,j}$ of type $H_i$ 
for $\Pi$, that is $\omega=\theta_{i,j}(c)$ (equivalently
$(\omega,\xi(c))=\hat{\theta}(c,j)$).
Suppose also
that $(c,{\sf w},c')$ is an edge of $H_i$, and 
$\omega' = \omega \cdot P({\sf w})$.
Then $\omega'$ is in the image of the same handle $\theta_{i,j}$,
for $\Pi$, with $\omega'=\theta_{i,j}(c')$,
and hence
$(\omega',\xi(c'))=\hat{\theta}(c',j)$.
\end{lemma}
In the following, where
we shall often need to refer to more than one brick, 
in order to specify a particular brick $B$,
we shall denote by $\Omega_B$ and $\hat{C}_B$ the sets $\Omega$ and $\hat{C}$ above,
by $(h_{B,i})$ the sequence of integers $(h_i)$, 
by $\theta_{B,i,j}$, $\hat{\theta}_B$ the 
handles $\theta_{i,j}$ and the bijection $\hat{\theta}$, and by $\Pi_B,P_B$ the partial permutation
representations $\Pi,P$.

\begin{remark}\label{shape}
We call the pair $(|\Omega_B|,(h_{B,i}:i=1,\ldots,\Nht))$ the {\bf shape} of the brick $B$;
it is knowledge of the shapes of bricks associated with particular
jump data that allows us to decide the
possible degrees of mosaics that we can construct out of them.
But properties such as primitivity or multiple transitivity 
of the action on the mosaic
may depend on further properties of the bricks. 
\end{remark}

\subsection{The  brick finder algorithm}
\label{sec:brickfinder}

\paragraph*{\sc Input:} Jump data $\cJ =(C,J,S) $,  a presentation $G=\langle X \mid R \rangle $ 
such that $X = X^{-1}$ is explicitly closed under inversion, 
a bound $\NN$ on the number of points in the partial permutation representation, 
and an initial $c\in C$. 
\paragraph*{\sc Output:} All bricks $P : X \to \PM(\Omega, \Omega \cup \hat{C}) $ compatible with the given 
jump data $\cJ$ such that $1\cdot P (\xi (c) ) = (c,1)$ and $|\Omega | \leq \NN $. 
\paragraph*{\sc Preprocessing:} We close the relators, jumps and stays under inversion and the
relators under rotation. Several sanity checks are done at this point. 

The algorithm will backtrack on the {\bf partial coset table} $\PCT$ for $P$,
defined to be the table whose rows are indexed by
elements $\{ 1,\ldots , N \}$ with $N\leq \NN$,
columns are indexed by the generators $x \in X$, 
where the entry in position $(\omega, x) $ of the coset table is given by
\[ \PCT(\omega,x) := \left \{ \begin{array}{ll} \omega' \in \Omega,& \mbox{or}\\
                   (c,j) \in C \times \N & \mbox{with } \xi(c)=x, \mbox{ or}\\
                    0 & \mbox{if not yet defined.}\\
\end{array} \right .
\]
Initially $N=1$ and the table is empty.
We shall prove below that 
if the partial coset table completes with $N$ rows, then 
$\PCT$ defines a brick $P$,
on $\Omega = \{ 1,\ldots,N\}$, 
with \[\omega \cdot P(x) := \PCT(\omega,x).\]
{\bf Remark:} 
The algorithm is a depth first backtrack search on the partial coset table, where at all 
times the rows stand for distinct points in $\Omega $. 
There is therefore no coincidence procedure in this algorithm. 
\\
Also we ensure at all times that if $\PCT(\omega,x) = \omega ' \in \Omega $ then $\PCT(\omega',x^{-1}) = \omega $.
\paragraph*{\sc Algorithm (summary):} 
Each partial coset table is first checked against the relators and the stays. 
If any contradiction is found the partial coset table is rejected. 
\\ 
Otherwise we `disjoin', that is,
a position $(\omega , x)$ is selected where $PCT(\omega , x) = 0$ 
and all possible values
(from $\Omega $ or $C \times \N$, fitting the constraints detailed below) 
are tried in turn and passed to the next level of the
backtrack. 
\paragraph*{\sc Some details of the algorithm:} 
\begin{itemize}
\item[(a)] The checking of a relator resembles that used in Todd-Coxeter \cite{ToddCoxeter}.
From each $\omega \in \Omega$, each $r \in R$ we trace out the successive
images of $\omega$ under the successive generators of $r$
until either we reach an undefined entry or a cement point
 before the end of $r$, or we return to
$\omega$ at the end of $r$.
 The point is then similarly traced backwards using the inverses of the
generators and if these two processes meet we check that it is the same in both directions.  
\item[(b)] For each stay $(c,{\sf w},c')$, and for
each $\omega \in \Omega$, $x \in X$ with $\PCT(\omega,x)=(c,j)$ for some $j$,
we trace out the successive
images of $\omega$ under the successive generators of ${\sf w}$
until either we reach an undefined entry before the end of ${\sf w}$ or we 
reach the end of ${\sf w}$, and check that all those entries that are defined
are in $\Omega$, and that the image under ${\sf w}$, if defined, is
an element $\omega' \in \Omega$ for which $\PCT(\omega,\xi(c'))=(c',j)$.
Under certain circumstances it would be possible to trace a stay backwards also.
\item[(c)]
For correctness it is sufficient to disjoin next at any position which ensures termination. 
For efficiency it is preferable to select a position, if possible, where the number of
continuations is minimized. If either of the above checks nearly completes that is 
a good position to disjoin. 
\item[(d)] 
The 
 possible choices for the entry in the selected position $(\omega ,x)$ are

$\bullet $
 any 
point $\omega' \in \Omega $ that is already in the partial coset table
for which $\PCT(\omega ' , x^{-1} ) = 0$. 
In this case we set both 
$\PCT(\omega, x) = \omega'$ and $\PCT(\omega ' , x^{-1} ) = \omega$. 

$\bullet $
a new point, $\omega'=N+1$, which we now add to $\Omega$ (so long as $N+1\leq \NN$). 
In this case too we set both 
$\PCT(\omega, x) = \omega'$ and $\PCT(\omega ' , x^{-1} ) = \omega$. 

$\bullet $
a cement point $(c,j)$ with $\xi(c) = x$, where $j$ already occurs in the 
table but $(c,j)$ does not 

$\bullet $
a new cement point $(c,j)$ with $\xi(c) = x$, where $j$ does not occur in 
the table but (assuming $j>1$) $j-1$ does. 
\end{itemize} 

\begin{remark}
(a) One of us has implemented this algorithm as a standalone c-program 
and it works quite well finding
thousands of bricks. We group them according their shape (see Remark \ref{shape}) and the number of fixed points of the generators and only output one 
example for each set of invariants. 
\\
(b) 
The brick finder algorithm can modified for use as a low index procedure, 
just by not giving an initial cement point. 
Then the program finds all transitive permutation representations
on less than $\NN $ points. \\
(c) 
It is good to know whether during processing the maximum point test 
was actually used. 
If not, it is clearly pointless to rerun the program with an increased bound $\NN $.
\end{remark}

\begin{theorem}\label{brickfinderTheorem}
If the partial coset table completes with $N$ rows, then 
$\PCT$ defines a brick $P:X\rightarrow \PM(\Omega,\Omega \cup \hat{C} ) $,
where $\Omega = \{ 1,\ldots,N\}$, 
and \[\omega \cdot P(x) := \PCT(\omega,x).\]
\end{theorem}
\begin{proof}
If the algorithm terminates,
on termination the $\PCT$ contains no null entries, and any entry
is either from $\Omega = \{1,\ldots,N\}$ of of the form $(c,j)$, with
$c \in C, j \in \N$.

The checks on relations, together with the fact that new images
$\omega \cdot P(x)$ within $\Omega$ are defined in pairs,
 ensure that the map $\Pi$ derived from $P$
by defining 
\[ \omega \cdot \Pi(x) = \left \{ \begin{array}{ll} \omega \cdot P(x) & \mbox{if that image is in } \omega \\
\void & \mbox{otherwise}  \end{array} \right.
\]
yields a partial permutation representation of $G$.
It is transitive by construction.

The selection of cement point entries ensures that no cement point
$(c,j)$  occurs twice,
and the check on stays ensures that when $c,c'$ are in the same handle type, then $(c,j)$ is an entry
precisely when $(c',j)$ is an entry. Hence we identify from the $\PCT$ integers
$h_i$, a set $\hat{C}$, and a bijection $\hat{\theta}$ from $\hat{C}$ to
$\{ (\omega,x): \PCT(\omega,x)=0\}$, as in B2. It is straightforward to
recover the handles $\theta_{ij}$ from $\hat{\theta}$.
\end{proof}

\subsection{The construction of the mosaic}
\label{sec:mosaic}

Now $G$ and jump data $\cJ$ for $X$ are specified, and we assume that 
$\cJ$ is compatible with $R$.
We suppose also that we have a finite
 set of bricks ${\cB}$ giving rise to mappings
$$P_B : X \to \PM (\Omega _B , \Omega _B \disj \hat{C}_B ) )  \mbox{ for all } B\in {\cB}.$$
These are the bricks that we are allowed to use in our mosaic. 
Usually we will use several copies of the same brick. 
To distinguish formally between the copies we choose an index set 
$\II$ together with a 
function
$\beta : \II \to  {\cB} $, so that $\beta (\ii ) = B$ means that 
we will use a copy of brick $B\in {\cB}$ in position $\ii $.

The aim of this section is to construct a permutation representation 
$\phi$ of 
$G$ on the set 
$$\cM := \bigcup _{\ii \in \II } \Omega _{\ii },\quad\mbox{where}\quad 
\Omega _{\ii } := \{ (\omega ,\ii ) 
\mid \ii \in \II , \omega \in \Omega _{\beta (\ii )} \} .$$ 
%

For each $1\leq i \leq \Nht$  we define 
$$D_i :=\{ (\ii, j ) \in \II \times \N  \mid 0<j\leq h_{\beta(\ii),i}\}; $$
the set $D_i$ indexes
the set of all handles of type $H_i$ within the mosaic. 
For $c\in C_i$ we define $D(c) := D_i$.

\begin{defn}
\label{defn:constructiondata}
We can define {\bf construction data} to be 
 a set 
$${\cZ}(C) = \{ \zeta (c) : D(c) \to D(\overline{c})  \mid  c\in C \} $$
of bijections, with 
$\zeta (\overline{c}) = \zeta (c)^{-1} $,
such that
the $\zeta (c)$ 
satisfy the groupoid relators ${\frak r}\in {\frak R}$, that is,
provided that the construction instruction  \[\zeta: \cj(c) \mapsto \zeta(c)\]
induces a groupoid homomorphism
from ${\cG}(\cJ,{\frak R})$ to ${\cZ}(C)$.
We call $\zeta$ the {\bf construction instruction}.
\end{defn}

We note that the map $\zeta$ associated with construction data
defines an action 
of the jump groupoid ${\cG}(\cJ,{\frak R})$ on the set
 \[\cD_\beta := \bigcup _{i =1}^\Nht D_i,\]
of all handles occurring in bricks in our mosaic. 
We use that action to connect our bricks into a mosaic.

\begin{defn}\label{mosaic}
For given construction instruction $\zeta$, we define the associated
{\bf mosaic} to be the pair $(\cM,\phi)$,
where
$${\cM}:= \bigcup _{\ii \in \II} \Omega_\ii \mbox{ with } \Omega_\ii := \{ (\omega , \ii) : \omega \in \Omega _{\beta (\ii)} \} ,$$
and
$\phi : X \to \Sym ({\cM} ) $ is defined by 
\begin{eqnarray*}
(\omega , \ii) \cdot \phi (x) &:=& \left\{ \begin{array}{ll} 
(\omega \cdot P_{\beta (\ii)} (x) ,\ii) & \mbox{ if } \omega \cdot P_{\beta (\ii)}(x) \in \Omega_{\beta(\lambda)} \\
(\omega ', \ii' ) & \mbox{ if }\omega \cdot P_{\beta (\ii)} (x) =:(c,j) \in \hat{C}_{\beta(\ii)},
\end{array} \right. 
\end{eqnarray*}
where in the second case, we 
define $\omega',\ii'$ via
$$ (\ii',j') := (\ii,j)\cdot \zeta(c) \mbox{ and }
(\omega',x^{-1} )  := \hat{\theta}_{\beta(\ii')}(\overline{c},j') .$$
\end{defn}

\begin{theorem}\label{main}
$\phi $ defines a group homomorphism $G\to \Sym ({\cM} ) $.
\end{theorem}
\bew
We first note that for all $x\in X$ and $(\omega ,\ii ) \in {\cM}$ 
$$(\omega ,\ii ) \cdot \phi (x) \cdot \phi (x^{-1}) = (\omega ,\ii ). $$
This is clear if $\omega \cdot P_{\beta(\ii)} (x) \in \Omega _{\beta(\ii )} $. 
Otherwise let
$(c,j) := \omega \cdot P_{\beta(\ii)} (x) = (\hat{\theta}_{\beta(\ii)})^{-1}(\omega,x)$. 
and let
$(\omega ' , \ii ') := (\omega ,\ii) \cdot \phi(x)$.
Then
$(\ii ,j) \cdot \zeta (c)  =
(\ii ', j' )  
\in D(\overline{c}) $ and 
so 
\[ (\ii,j)= (\ii ', j' )\cdot \zeta({c})^{-1}=  (\ii ', j' )\cdot\zeta(\bar{c}),\]  
while also $\overline{\overline{c}} =c$.
Further
$(\omega',x^{-1} ) = \hat{\theta}_{\beta(\ii')}(\overline{c},j')$,
and so by definition (of the codomain of the bijection $\hat{\theta}_{\beta(\ii')}$ (see Remark~\ref{rem:brickdefn}(c)),
$\omega' \cdot P_{\beta(\ii')}(x^{-1})\not\in \Omega_{\beta(\ii')} $
Hence it follows from the equation $\hat{\theta}_{\beta(\ii)}(c,j) = (\omega,x)$
that
$$(\omega ' , \ii ')  \cdot \phi(x^{-1}) =  (\omega ,\ii ). $$

It remains to show that $\phi $ satisfies the relators in $R$.
Given a relator
$r=(x_1 \cdots  x_t)^m \in R$ 
and some point $(\omega , \ii) \in {\cM}$ we need to show that $(\omega,\ii) \cdot \phi(r) = (\omega,\ii)$.
Since $\Pi_{\beta(\ii)} $ defines a partial 
permutation representation we have $\omega \cdot P_{\beta (\ii)}(r) = \omega $ if the 
successive application of the maps $P_{\beta (\ii)} (x_i)$ to $\omega $ keeps its images within the set $\Omega _{\beta (\ii)}$.
If not, then there is a first position $i_1 \in \{ 1,\ldots , t\} $ and some $\omega '$ 
in the sequence of successive images 
 such that 
$\omega ' \cdot P_{\beta(\ii)} (x_{i_1}) \in \hat{C}_{\beta(\ii)}$.
Suppose that
$\omega ' \cdot P_{\beta(\ii)} (x_{i_1}) = (c_1,j_1)$  for some $c_1,j_1$
In this case, $\omega'$ is within the image of a handle of
the brick $\beta(\ii)$ 
and hence $(\omega',x_{i_1})=\hat{\theta}_{\beta(\ii)}(c_1,j_1)$.
It is enough to 
show that for the rotated relator 
$r'=(x_{i_1}\ldots )^m$
we have $(\omega',\ii) \cdot \phi(r') = (\omega',\ii)$.
The compatibility of $\cJ$ with $R$ ensures that $r'$ factorises
as
$$r' = x_{i_1} {\sf w}_1 x_{i_2} {\sf w}_2\ldots  x_{i_s} {\sf w}_s$$
where,
for each $1\leq j \leq s $, 
we have $x_{i_j} = \xi(c_j)$ for some $c_j \in C$,
and $ (\overline{c_j} , {\sf w}_j, c_{j+1} ) \in S $
(addition $\bmod\,s$), 
and hence we see that $\cj(c_1)\cdots \cj(c_s)=1$ holds in $\cG$.
Put $\omega _1 := \omega ' $, $\ii_1 := \ii$.
For $n=1,\ldots , s $, put 
$$(\ii_{n+1},j_{n+1}) := (\ii_{n},j_{n}) \cdot \zeta (c_{n}) .$$
Then $(\omega _1,\ii_1) \cdot \phi (x_{i_1}) = (\omega _1' , \ii_2) $ where 
$(\omega _1',x_{i_1}^{-1}) =\hat{\theta}_{\beta (\ii_2)}(\overline{c_1},j_2)$. 
Lemma~\ref{lem:brick} ensures that
$(\omega _1' , \ii_2) \cdot \phi ({\sf w}_1) = (\omega _2 , \ii_2 )
\in \Omega_{\beta(\ii_2)}$,
and also that $(\omega_2,x_{i_2}) = \hat{\theta}_{\beta(\ii_2)}(c_2,j_2)$.
We continue like this (repeatedly applying Lemma~\ref{lem:brick})
and hence compute 
$(\omega',\lambda)\cdot \phi(r')$
as
$$(\omega _s',\ii_{s+1} ) \cdot \phi({\sf w}_s) = (\omega _{s+1},\ii_{s+1} )  
\in \Omega_{\beta(\ii_{s+1})},$$
and see that 
$(\omega_{s+1},x_{i_1}) = \hat{\theta}_{\beta(\ii_{s+1})}(c_1,j_{s+1})$. 
 
Since $\zeta (c_1) \cdots \zeta (c_s) = 1$ we get $(\ii_{s+1},j_{s+1})= (\ii_1,j_1)$ and so
\\
$(\omega_{s+1},x_{i_{s+1}}) = \hat{\theta}_{\beta(\ii_1)}(c_1,j_1)=(\omega_1,x_{i_1})$.
\eb

\begin{remark} \label{transitivity} 
As all the constituent bricks are transitive,
the action defined by $\phi$ on $\cM$ is
transitive,
provided that the graph on $\Lambda$ for which 
$$\lambda \sim \lambda' \iff \exists j,j',c, (\lambda,j)\cdot \zeta(c) = (\lambda',j')  $$
is connected.
\end{remark}

There might be an algorithm to find all mosaics made from bricks of given shape. 
However in practice the groupoid relators are not too complicated 
and families of mosaics can be constructed by hand, as we will illustrate in the next section.

\section{Some examples} 
\label{sec:examples}

\subsection{Example 1}
Let $$G := \langle s,t \mid s^3, t^2, (st)^7 \rangle $$
So in our notation $X=\{ s,t,r \} $ with $s^{-1} = r$, $t^{-1} = t$, 
$$R= \{ s^3 , r^3, (st)^7, (ts)^7, (rt)^7, (tr)^7 \} $$
We choose $C=\{ c_1,c_2 = \overline{c_1} \} $, with $\xi(c_1)=t=\xi(c_2)$.
$$\begin{array}{l}
J=\cj (C):= \{ (c_1,t,c_2) , (c_2,t,c_1 ) \} , 
\\  S = \{ (c_1, ststs, c_1), (c_2,stststs, c_2),
(c_1, rtrtr, c_1), (c_2,rtrtrtr, c_2) \} \end{array} $$
Then the equivalence classes in $C$ defined by $\sim_S$ are $[c_1] = \{ c_1 \} $ and $[c_2] = \{ c_2 \}$ and
the inv-tab groupoid ${\cF}(\cJ) $ is generated by $\ _1 \cj(c_1) _2 $, $\ _2 \cj(c_2) _1 = (_1\cj(c_1)_2)^{-1} $.
We see that $t$ is the only cementable generator; so we examine factorisations
of $(ts)^7$ and $(tr)^7$ in order to verify that $\cJ$ is compatible with $R$
and find the relations ${\frak R}$ that define the jump groupoid $\cG(\cJ,{\frak R})$ from $\cF(\cJ)$. 
We see that $(ts)^7$ has a compatible factorisation as
$\xi(c_1)\,stststs\,\xi(c_2)ststs$
 (since $(\overline{c_1},stststs,c_2)$ , $(\overline{c_2},ststs,c_1) \in S$), and
$(tr)^7$ a compatible factorisation as
$\xi(c_1)\,rtrtrtr\,\xi(c_2)\,rtrtr$, giving
\[ {\frak R} = \{\cj(c_1)\cj(c_2)\} = \{ \cj(c_1)\cj(\overline{c_1}) \}.\]
So we have ${\cG}(\cJ,{\frak R}) = {\cF}(\cJ)$.

The brick finder algorithm was run with a limit of 110 points finding many thousands of bricks,
of which we shall use just two, one on 28 points and one on 57 points, each with one
handle of each type. 
Let $\cB = \{B_1,B_2\}$ be that set of two bricks, $B_1$ of shape $(28,(1,1))$ 
and $B_2$ of shape $(57,(1,1))$. 
Choose $\II=\{1,\ldots,M\}$ for any integer $M$ and any  surjection
$\beta:\II \rightarrow \cB$. The sets $D_1=D(c_1)$ and $D_2=D(c_2)$ are both
in correspondence with $\II$ and so can be identified with $\II$, and the 
map $\zeta$ defined by
\[ \zeta(c_1):\ii \mapsto \ii+1,\quad \zeta(c_2);\ii \mapsto \ii-1 \]
(addition defined $\bmod M$) is a valid construction instruction.

\begin{wrapfigure}{r}{0.4\textwidth}
\begin{center}
 \includegraphics[scale=0.5]{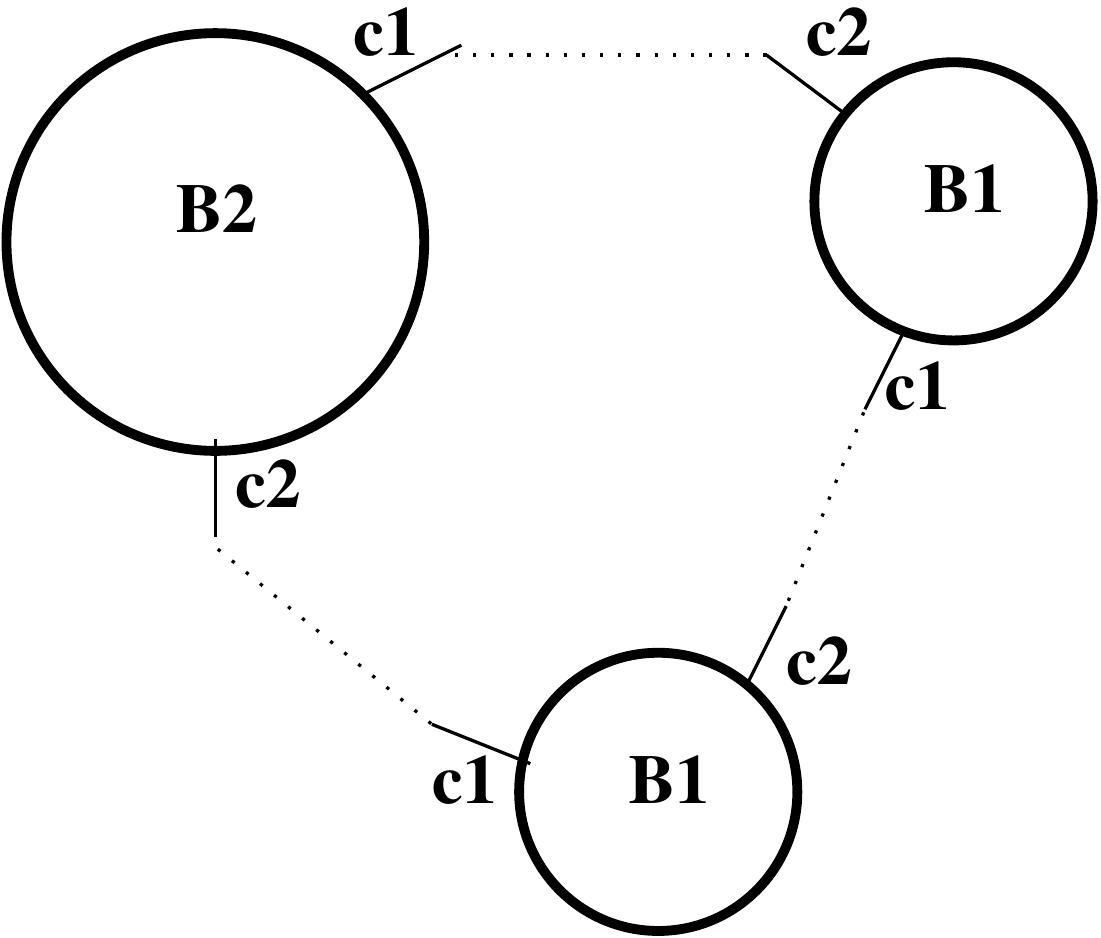}
\end{center} 
\end{wrapfigure}

Informally, this means we can fit the two bricks together in any combination by joining them up in a circle of $M$ bricks, for example  the picture to the right
yields a transitive permutation representation of $G$ on 113 points 
which, $113$ being prime,
 is obviously primitive, and since the group order is divisible by 
2,3, and 7, its image is clearly the alternating group.
The next theorem is an easy consequence of the fact that
we may join these two bricks in an arbitrary circle to obtain a 
transitive permutation representation of $G$.
\begin{theorem}
For any $n \in \N$ with $n\geq 28 \cdot 57 $ the 
group $G = \langle s,t \mid s^3, t^2, (st)^7 \rangle $ has a transitive permutation 
representation on $n$ points. 
\end{theorem}


Using more bricks one can decrease the bound $28 \cdot 57 $ considerably, but this problem has
already been solved by \cite{Stothers3}, which states that the true bound is 168. 
We will come back to this for other triangle groups in Example 3.

\subsection{Example 2} 

Take $G$ as above and define 
\begin{eqnarray*} 
C &:=& \{c_1,c_2=\overline{c_1},c_3,c_4=\overline{c_3}\},\,\xi(c_1)=\xi(c_3)=s,\,\xi(c_2)=\xi(c_4)=r,\\
\cj(C)&:=&\{ (c_1,s,c_2), (c_2,r,c_1), 
 (c_3,s,c_4), (c_4,r,c_3)  \} , \\ 
S&:=& \{ (c_2,\emptyword,c_1),(c_1,\emptyword,c_2), (c_3,\emptyword,c_4), (c_4,\emptyword,c_3), 
(c_2,tstst,c_3) , (c_3,trtrt,c_2) ,\\
&&\quad (c_4,tststst,c_1), (c_1,trtrtrt,c_4) \},
\end{eqnarray*}
where $\emptyword$ denotes the empty word.
Now $r$ and $s$ are the cementable generators, and we need to
check for compatible factorisations of $r^3,s^3,(rt)^7$, and $(st)^7$.
We factorise $r^3$ as $\xi(c_2)^3$ (note that $(\overline{c_2},\epsilon,c_2) \in S$) and as $\xi(c_4)^3$ and
$s^3$ as $\xi(c_1)^3=\xi(c_3)^3$.
Then we factorise
$(rt)^7$ and $(st)^7$ as 
\begin{eqnarray*}
(rt)^7 &=& \xi(c_4)\,trtrt\,\xi(c_2)\,trtrtrt\quad \ ; \ (\overline{c_4},trtrt,c_2),(\overline{c_2},trtrtrt,c_4)\in S\\
(st)^7 &=& \xi(c_1)\,tstst\,\xi(c_3)\,tststst\quad\ ; \ (\overline{c_1},tstst,c_3),(\overline{c_3},tststst,c_1) \in S.\\
\mbox{Then }{\frak R} &=& \{ \cj(c_2)^3,\cj(c_4)^3,\cj(c_1)^3,\cj(c_3)^3,
\cj(c_4)\cj(c_2), \cj(c_3)\cj(c_1) \} \end{eqnarray*} 
Since all the pieces of cement are equivalent, the groupoid ${\cG}(\cJ,{\frak R})$ is 
a group: 
$${\cG}(\cJ,{\frak R}) = \langle \cj(c_1),\cj(c_2),\cj(c_3),\cj(c_4) \mid 
\cj(c_1)\cj(c_2), \cj(c_3)\cj(c_4), \cj(c_1)^3, \cj(c_3)^3, \cj(c_1) \cj(c_3) \rangle $$
So clearly ${\cG}(\cJ,{\frak R}) \cong C_3$, which means that any transitive mosaic
can have either one or three bricks. 

We find six brick shapes with $|\Omega | \leq 40$ points containing this handle 
exactly once: $|\Omega | = 14,21,28,29,35,36 $, any 3 of those can be joined 
in a circle to construct a permutation representation of $G$. 
Gluing the brick $B$ on 14 points 3 times in a circle we obtain  an imprimitive permutation
representation $\varphi $ on $3\cdot 14$ points.

\vspace*{-7cm}
$
\begin{array}{r|r|r|r} 
B & s & r & t \\ 
\hline
1 &  c_1& c_2&  2 \\
2 &  3 & 4 & 1 \\
3 &  4 & 2 & 5 \\
4 &  2 & 3 & 6 \\
5 &  7 & 8 & 3 \\
6 &  9 &10 & 4 \\
7 &  8 & 5 &11  \\
8 &  5 & 7 &10 \\
9 & 10 & 6 &12 \\
10&   6&  9&  8 \\
11&  c_3& c_4&  7 \\
12&  13& 14&  9 \\
13&  14& 12& c_5 \\
14&  12& 13& c_6 
\end{array} $
\hfill
 \includegraphics[width=8cm,height=11cm,trim=0 200 0 -400bp]{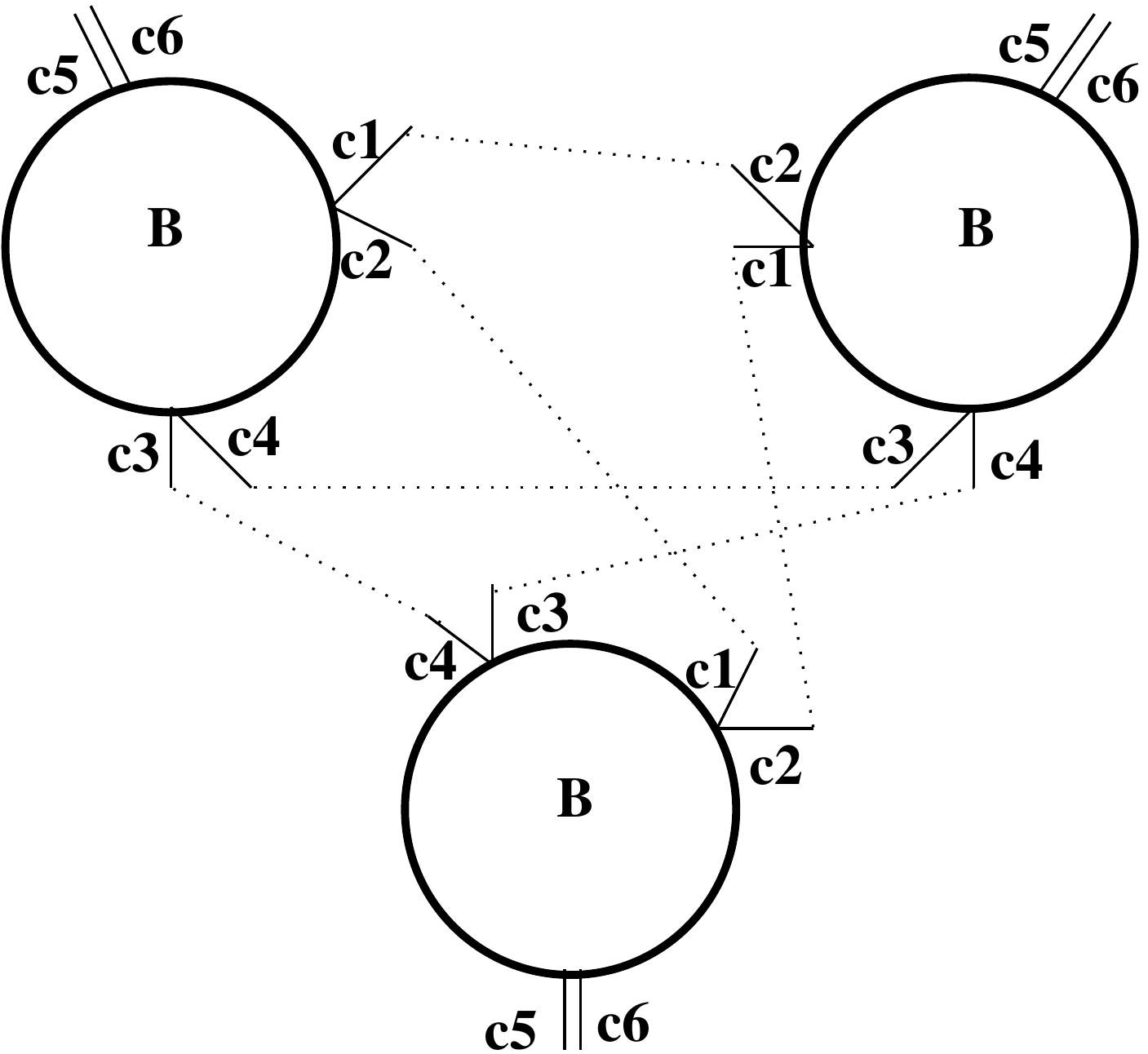}

\vspace*{0.5cm}

We now add 
more cement,  $c_5=\overline{c_5}$ and $c_6=\overline{c_6}$,
with $\xi(c_5)=\xi(c_6)=t$,
and put
$$J' := \cj(C) \cup \{ (c_5,t,c_5), (c_6,t,c_6) \} ,
S':= S \cup \{ (c_5,s,c_6), (c_6,r,c_5), (c_6,s(ts)^5,c_5), (c_5,r(tr)^5,c_6) \} .
$$
Now $(ts)^7$ factorises as $\xi(c_5)\,s\,\xi(c_6)\,s(ts)^5$ and
$(tr)^7$ factorises as $\xi(c_6)\,r\,\xi(c_5)\,r(tr)^5$, so
that the extra relators that define $\cG$ as a quotient of $\cF(\cJ)$ are
given by
$${\frak R}' := {\frak R} \cup \{ \cj(c_5)\cj(c_6)\}$$. 

Restricting to brick shapes on less that 40 points we find further bricks:
one brick $B$ of shape $(14,(1,1))$ (as displayed in the 
coset table above) 
bricks $A_i$  of shape $(i,(0,1)) $ 
for  $i= 7,14,15,21,22,28,35,36$ points 
and one brick of shape $(28,(0,2))$. 

\begin{wrapfigure}{r}{0.4\textwidth}
\begin{center}
 \includegraphics[scale=0.6]{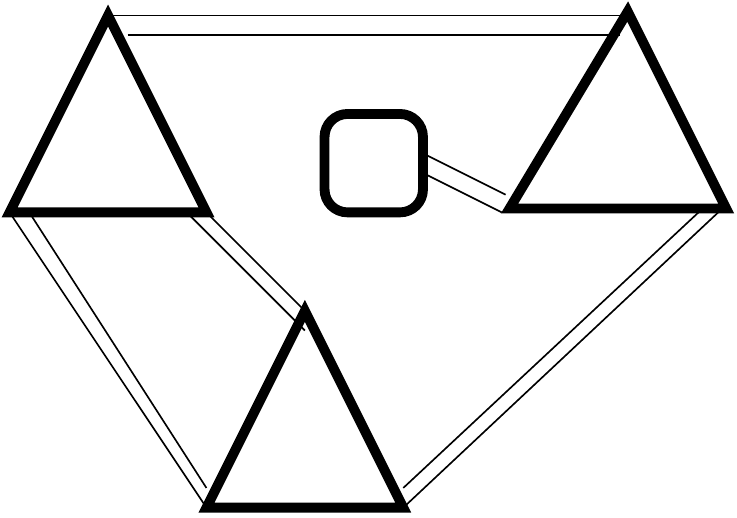}
\end{center} 
\end{wrapfigure}
The imprimitive permutation representation $\varphi $ 
on $3\cdot 14$ points from above can also be seen as one brick $\Delta$ of 
shape $(42,(0,3))$ (with cement points on the fixed points of $t$, 
one $c_5$ and $c_6$ in each of the 3 imprimitivity regions) as in the 
picture above. This brick $\Delta $ 
will be visualized as a triangle in the picture on the right. 
Taking 
one copy of the brick of shape $(7,(0,1))$ (the rounded square in the
picture) we construct a mosaic on
$9 \cdot 14 + 7 = 133 $ points using the new cement as indicated in the
picture on the right.
The resulting permutation group $\phi (G)$ is indeed isomorphic to the
alternating group of degree 133.

\begin{remark} 
Note that the action of 
the jump groupoid $\cG (\cJ' , {\frak R}') $ on the
handles in the mosaic is intransitive, there are no
jumps between  the two handle types $H_1 = \{ c_1,c_2,c_3,c_4 \}$ and $H_2 = \{ c_5,c_6 \}$.
Nevertheless the constructed mosaic is transitive, 
because all handles are connected by a sequence of jumps and 
bricks (see Remark \ref{transitivity}). 
\end{remark}

\subsection{Example 3: 
Some (2,3,n) triangle groups.} 

In \cite{Stothers3} the (2,3,n) triangle groups are considered.
So let 
$$G(2,3,n) := \langle s,t \mid s^3, t^2, (st)^n \rangle $$
Then Stothers defines 
$$M(n) := \min \{ N\in \N \mid G(2,3,n) \mbox{ has a subgroup of index } m \mbox{ for all }
t\geq N \}$$
and determines bounds on $M(n)$ (\cite[Theorem 1]{Stothers3}): 
$$\begin{array}{|c|c|c|c|c|c|c|c|c|c|c|} 
\hline 
n & 7 & 8 & 9 & 10 & 11 & 12 & 13 & 14 & 15 & 16 \\
\hline
M(n) & 168 & \leq 240 & \leq 180 &  10 & \leq 110 & \leq 240 & \leq 143 & \leq 154 & 
\leq 210 & \leq 128 \\
\hline 
\end{array}
$$

Using our technique we are able to compute the exact values of $M(n)$ for certain $n$:
\begin{theorem}
$M(8) = 24$, $M(9) = 35$, $M(12) = 12$, $M(14) = 14$, $M(15)= 15$, $M(16) = 12 $. 
\end{theorem}
We did not find improvements for $n=11$ and $13$, however.

\bew
The method of proof is always the same:  \\
For odd $n$ we use $C=\{c_1,c_2 \}$ and the following jump data for $G(2,3,n)$:
\begin{eqnarray*}
\cj (C) &:=& \{ (c_1,t,c_2) , (c_2,t,c_1 ) \} ,\\
S &:=& \{ (c_1, (st)^{(n-3)/2} s, c_1), (c_2,(st)^{(n-1)/2} s, c_2),
(c_1, (rt) ^{(n-3)/2} r, c_1), (c_2,(rt)^{(n-1)/2 }r, c_2) \} 
\end{eqnarray*}
As in example 1 we have two different handle types $H_1$ and $H_2$ and 
the jump groupoid is isomorphic to the one of example 1.
For $i=1,2$ let 
\begin{eqnarray*}
A_i &:=& \{ |\Omega | \mid \mbox{ there is a brick on $\Omega $ with a single handle, which has type } H_i \} \\
A_{12} &:=& \{ |\Omega | \mid \mbox{ there is a brick on $\Omega $ with two handles, one of each of the types } H_1, H_2 \}
\end{eqnarray*}
For even $n$, however, we use the jump data 
$$\cj (C) := \{ (c_1,t,c_1) \} ,  S = \{ (c_1, (st)^{(n-2)/2} s, c_1)
(c_1, (rt) ^{(n-2)/2} r, c_1) \} $$
and get just one handle type ($H_1$). 
As we just have one handle type, the groupoid is a group; 
The inv-tab groupoid $\cF (\cJ ) $ is isomorphic to the 
infinite cyclic group, the jump groupoid $\cG (\cJ ,{\frak R} )$ is isomorphic
to the group of order 2. 
In this case we define 
\begin{eqnarray*}
A_1 &:=& \{ |\Omega | \mid \mbox{ there is a brick on $\Omega $ with a single handle}\} \\
A_{12} &:=& \{ |\Omega | \mid \mbox{ there is a brick on $\Omega $ with two handles} \} 
\end{eqnarray*}
For convenience of notation we define $A_2:=A_1$ in this case.

In that case, whether $n$ is even or odd,
out of these bricks
 we may construct a transitive permutation representation on 
$N$ points provided that  
$$N \in \{ a+b+ \sum _{j=1}^s c_j \mid a \in A_1, b\in A_2, s\in \N_0, c_j \in A_{12} \} 
\cup \{ \sum _{j=1}^s c_j \mid  s\in \N , c_j \in A_{12} \} .$$
Let $c:= \min (A_{12}) $ be the minimum value in $A_{12}$.
As we may always add multiples of $c\in A_{12}$ it is enough to look at the 
smallest number in each congruence class modulo $c$ to get a (usually good) 
upper bound $b$ for $M(n)$. 
To compute the exact value, we need enumerate the subgroups of
index $<b$ in $G(2,3,n)$, which can be done with the brick finder algorithm 
by searching for bricks with no handles (or any other low index procedure). 
\\
In the example $n=12$, we find 
$$A_1 \supset \{ 6,7,9,10,12,13 \} , A_{12} \supset  \{ 12,15,18,21,22 \} .$$
To show that $G(2,3,12)$ contains a subgroup of index $N$ for all $N\geq 12$
we write all numbers from 12 to 23 that are not in $A_{12}$ 
in the form $a+b+c$ with $a,b\in A_1, c \in A_{12}$. 
Using low index it is easy to see that there is no subgroup of index 11.
So $M(12) = 12 $. 
\eb

\subsection{Example 4: A hyperbolic reflection group.}

This last example illustrates the power of the brick method.
It allows to construct infinite series of permutation representations where 
we may prove that the image contains  the full alternating group.

\begin{wrapfigure}{r}{0.3\textwidth}
\begin{center}
 \includegraphics[scale=0.6]{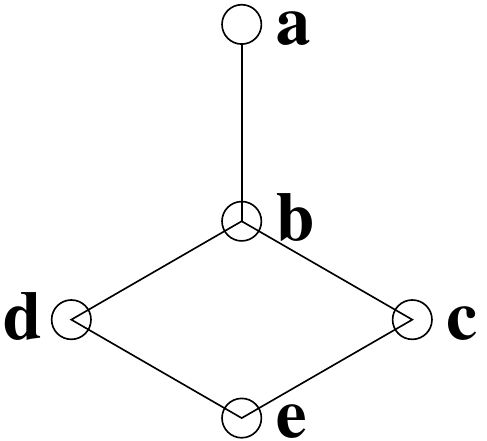}
\end{center} 
\end{wrapfigure}

As it is well known that almost all triangle groups have almost all alternating groups
as a quotient \cite{Everitt2} 
we will choose a different interesting group: 
Let $H$ be the hyperbolic reflection group, whose presentation is
given by the Dynkin diagram on the right.

\begin{theorem}\label{altquot}  
The  group 
$$H = \left\langle \begin{array}{l} 
a,b,c,d,e   \mid a^2,b^2,c^2,d^2,e^2, \\ (ab)^3, (ac)^2, (ad)^2, (ae)^2, (bc)^3,(bd)^3,(be)^2, (cd)^2,(ce)^3,(de)^3 \end{array} \right\rangle $$ 
has a permutation  representation on $N$ letters whose 
image contains the alternating group of degree $N$ for any $N\geq 703 $.
\end{theorem} 

For the proof of this theorem we need the following easy lemma.

\begin{lemma}\label{alternating} 
Let $G = \langle g_1, g_2,\ldots ,  g_m \rangle$
 be a transitive permutation group on some finite set $\Omega $
 containing a subgroup $U\leq G$ that is the
alternating group acting naturally on a subset $Y\subset \Omega $ and 
as the identity on $\Omega \setminus Y$. 
If $|Y| \geq 5$ and each of the generators $g_i$ takes at least one point of 
$Y$ to a point of $Y$, then $G$ contains the alternating group on $\Omega $.
\end{lemma}

\bew
We proceed by induction on $|\Omega \setminus Y|$, where the 
statement is clear, if $Y=\Omega $. 
So assume that $Y \neq \Omega $. 
Since $G$ is transitive, there is a generator $g_j$ and some $\omega \in Y$ 
such that  $\omega \cdot g_j \not\in Y$.
Put $V:=\langle U, U^{g_j} \rangle $ and $Y' := Y\cup (Y\cdot g_j) $. 
Then clearly $V$ fixes all points in $\Omega \setminus Y'$ and the 
lemma follows by induction if we show that 
$V$ acts as the alternating group of degree $|Y'|$ on $Y'$. 
\\
As $U$ acts primitively on $Y$ and $|Y'| < 2|Y|$ 
the action of $V$ on $Y'$ is primitive (see also
\cite[Proposition 8.5]{Wielandt}). 
The group $U$, being the alternating group of $Y$, contains a 
prime cycle of length $p \leq |Y'| - 3$. 
Then a theorem by Jordan (1873) (\cite[Theorem 13.9]{Wielandt}) 
shows that $V$ is either the alternating or symmetric group on $Y'$. 
\eb

\bew (of Theorem \ref{altquot}).
We define the following jump data: 
\begin{eqnarray*}
C&:=&\{c_1,c_2\} , \ J:= \{ \cj(c_1):=(c_1,e,c_2), \cj(c_2)=\cj(\overline{c_1}):= (c_2,e,c_1) \}\\
S&:=& \{ (c_1,w,c_1) , (c_2,v,c_2 ) \mid w\in \{ a,b,d, cec \} , v \in \{ a,b,c,ded \} \} .\end{eqnarray*}
This gives ${\frak R} = \{ \cj(c_1)\cj(c_2) \}$, so $\cG ( \cJ, {\frak R}) = \cF (\cJ ) $.

The bricks on less than 240 points all contain both cement points with the same multiplicity. 
We find bricks $$B,A_{40},A_{60},A_{185} ,$$ 
on $16,40,60,$ and $185$ points, where this multiplicity is 1.
As the jump groupoid is free on one element, these bricks can be 
combined arbitrarily in a circle, the circles of length one leading to 
permutation representations $P_i$ of $H$ of degree $i=16,40,60,185$, respectively. 

The first step in the construction of permutation representations of $G$ as 
alternating or symmetric groups of arbitrarily large degree is to construct
representations of 16 different degrees $n_0,\ldots,n_{15}>0$, where $n_i = i\,
\bmod{16}$.
For this we proceed as follows: 
\begin{itemize} 
\item[(1)] We need two words $w_1,w_2$ in the generators of $H$ such that $P_{16}(w_i) = 1$ 
and $P_k(w_i) \neq 1$ for $k=40,60,185$. 
\\
The words $w_1=(ebcbed)^2, w_2= (abcdec)^5 $ satisfy these conditions and
pass the later checks we need as well. 
\item[(2)] Put $w:=w_1w_2w_1^{-1} w_2^{-1} $. We check
 that $P_k(w) \neq 1$ for $k=40,60,185$. 
\item[(3)] For each $i =0,1.\ldots,15$, choose integers $n_i>0,a_i,b_i,c_i \geq 0$ with $n_i=40a_i+60b_i+185c_i= i \bmod{16}$, and $n_i$ as small as possible
subject to those constraints (the maximum value of $n_i$ is $3\cdot 185+60+40=655$).  
\item[(4)] For each $n_i$, construct a mosaic on $n_i+48$ points by gluing  3 copies $B_1,B_2,B_3$ of the brick $B$, then $a_i$ copies of $A_{40}$, $b_i$ copies of $A_{60}$ and $c_i$ copies of $A_{185}$ to a circle.  

Let $M_i$ be the union of the copies of $A_{40}$, $A_{60}$ and $A_{185}$ within 
the mosaic. We can use $M_i$ (two of whose original handles are connected in the mosaic to handles in the
two copies $B_1,B_2$, say, of $B$ that are next to it in the circle, the others to other handles within $M_i$) 
as a brick with just two handles in larger mosaics.
\item[(5)] We can check that the subgroup $U:=\langle w, w^a,w^b,w^c \rangle $ acts as
 the alternating group on $n_i+18$ points of the mosaic within $(B_1,M_i,B_2)$,  and fixes the other $30$ points. 
\end{itemize} 

\begin{wrapfigure}{r}{0.25\textwidth}
\begin{center}
\vspace*{-1cm} 
 \includegraphics[scale=0.5]{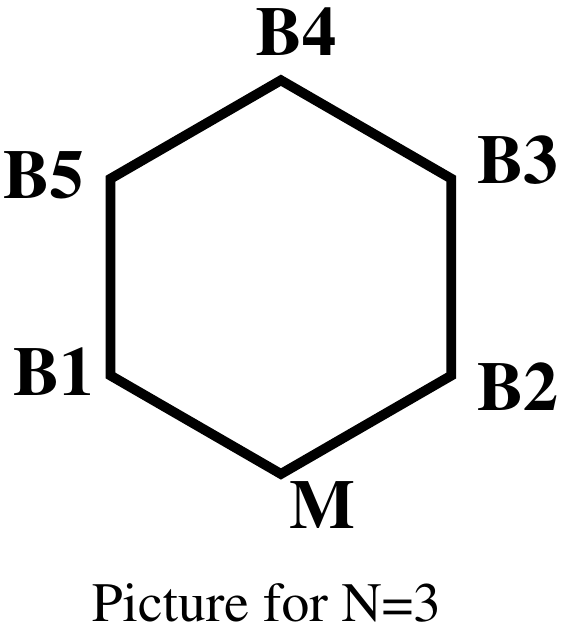}
\end{center}
\end{wrapfigure}

All the computations above were done in Magma \cite{Magma}. 

The second step of the construction is to use the 16 `bricks' $M_i$ of degrees $n_i$
to construct larger mosaics, as follows.  
For $N \in \N$, and some $i \in \{0,1,\ldots,15\}$,
 let $n=16(N+2) + n_i$.
We construct a mosaic $\phi$ on $n$ points by gluing
$N+2$ copies $B_1,\ldots,B_{N+2}$  of $B$ to the brick $M_i$ 
as shown below for $N=3$. When $N=1$, we get the mosaic defined in (4) above.

By Theorem \ref{main},
the mapping $\phi$ is a transitive permutation representation of $H$. 
As $w$ is a commutator of two elements in $H$ that fix all the points in $B$, 
$\phi(U)$ fixes all the $16N$ points of $B^N := (B_3,\ldots,B_{N+2})$.
As we have checked in (5) above $\phi (U) $ 
acts as the alternating group on a subset $Y_i$ of
$n_i+18$ points of $(B_1,M_i,B_2)$ and fixes all the other $16N+14$ points.
All generators $a,b,c,d,e$ of $H$ map some point of $Y_i$ to some other point of 
$Y_i$. 
So it follows from Lemma \ref{alternating}
that the image of $\phi$ contains the alternating group $A_{16(N+2)+n_i}$.
\eb


\end{document}